\documentclass{amsart}

\usepackage{eufrak}

\usepackage{amssymb,amsmath, amscd, mathrsfs, yfonts, bbm, graphicx}
\title[Non-crossing linked partitions ]
{Non-crossing linked partitions and multiplication of free random variables}

\author{Mihai Popa }

\address{Indiana University at Bloomington,
 Department of Mathematics, Rawles Hall,
 831 E 3rd St, Bloomington, IN 47405}
\email{mipopa@indiana.edu}

\DeclareMathAlphabet{\mathpzc}{OT1}{pzc}{m}{it}
\newtheorem{claim}{}[section]

\newtheorem{thm}[claim]{Theorem}

\newtheorem{remark}[claim]{Remark}
\newtheorem{prop}[claim]{Proposition}

\newcommand{\ca}{\mathcal{A}}

\newcommand{\kr}{\text{Kr}}

\newcommand{\cE}{\mathcal{E}}

\newcommand{\su}{\succeq}

\newcommand{\lra}{\longrightarrow}

\usepackage{amssymb}
\input xy
\xyoption{all}

\begin{document}
\begin{abstract}
 The material gives a new combinatorial proof of the multiplicative property of the $S$-transform. In particular, several properties of the coefficients of its inverse are connected to non-crossing linked partitions and planar trees.
 Operator Algebras).
\end{abstract}

 \maketitle
\bibliographystyle{alpha}

\section{Introduction and definitions}

 The relation between non-crossing partitions
and free probabilities has been studies extensively (see \cite{ns}, \cite{speicher1}), but the closely
related non-crossing linked partition have not received the same attention.
Recently (see \cite{dykema}, \cite{mpttransf}), the latest object was shown to give the recurrence for computing the coefficient of the inverse of the Voiculescu's $S$-transform in a similar manner the non-crossing partitions are used for the computation of the $R$-transform (see (\ref{cumulants}) and (\ref{tcoeff}) below). The present material gives a new, combinatorial proof of the multiplicative property of the $S$-transform using a relation between planar rooted trees and the Kreweras complement.

 A non-commutative probability space is a couple $(\ca, \phi)$, where $\ca$
 is a unital ($\ast$-)algebra and $\phi:\ca\lra\mathbb{C}$ is a linear mapping such that $\phi(1)=1$
 and $\phi(x^\ast x)\geq 0$ for all $x\in\ca$ if $\ca$ is a $\ast$-algebra. The  ($\ast$-)subalgebras $\{\ca_i\}_{i\in I}$ of $\ca$
 are \emph{free} if
 \[\phi(a_1a_2\cdots a_n)=0\]
 for any $a_k\in\ca_{\varepsilon(k)}$ such that $\varepsilon(j)\neq \varepsilon(j+1)$ and
 $\phi(a_k)=0$. The elements $\{X_i\}_{i\in I}$ from $\ca$ are free if the unital $\ast$-algebras they generate are free.

  A non-crossing partition $\gamma$ of the ordered set
 $\{1,2,\dots,n\}$ is a collection $C_1,\dots, C_k$
 of subsets of $\{1,2,\dots,n\}$, called blocks, with the following
 properties:
 \begin{enumerate}
 \item[(a)]$\displaystyle{\bigsqcup_{l=1}^kC_l=\{1,\dots,n\}}$ (disjoint union of sets)
 \item[(b)]$C_1,\dots,C_k$ are non-crossing, in the sense that there
 are no two blocks $C_l, C_s$ and $i<k<p<q$ such that $i,p\in C_l$
 and $k,q\in C_s$.
\end{enumerate}

\textbf{Example 1}: Below is represented graphically the
non-crossing
partition\\
 $\pi=(1,4,6), (2,3), (5),(7,8),(9,10),
(11,12)$ $\in NCL(10)$:

 \setlength{\unitlength}{.14cm}
 \begin{equation*}
 \begin{picture}(10,8)

\put(-16,3){\circle*{1}} \put(-16,3){\line(0,1){5}}

\put(-12,3){\circle*{1}} \put(-12,3){\line(0,1){3}}

\put(-8,3){\circle*{1}}\put(-8,3){\line(0,1){3}}

\put(-4,3){\circle*{1}}\put(-4,3){\line(0,1){5}}

\put(0,3){\circle*{1}}\put(0,3){\line(0,1){3}}

\put(4,3){\circle*{1}}\put(4,3){\line(0,1){5}}

\put(8,3){\circle*{1}}\put(8,3){\line(0,1){5}}

\put(12,3){\circle*{1}}\put(12,3){\line(0,1){5}}

\put(16,3){\circle*{1}}\put(16,3){\line(0,1){5}}

\put(20,3){\circle*{1}}\put(20,3){\line(0,1){5}}

\put(-16,8){\line(1,0){20}}

\put(-12,6){\line(1,0){4}}

\put(8,8){\line(1,0){4}}

\put (16,8){\line(1,0){4}}

\put(-16.5,0.5){1}

\put(-12.5,0.5){2}

\put(-8.5,0.5){3}

\put(-4.5,0.5){4}

\put(-0.5,0.5){5}

\put(3.5,0.5){6}

\put(7.5,0.5){7}

\put(11.5,0.5){8}

\put(15.5,0.5){9}

\put(19,0.5){10}

 \end{picture}
 \end{equation*}

 Non-crossing partitions appear in the definition of the free cumulants, the multilinear functions $\kappa_n:\ca^n\lra\mathbb{C}$ given by the recurrence ($X_1,\dots,X_n\in\ca$):
 \begin{equation}\label{cumulants}
 \phi(X_1\cdots X_n)=\sum_{\gamma\in NC(n)}\prod_{\substack{C=\text{block in} \gamma\\C=(i_1,\dots,i_l)}}
 \kappa_l(X_{i_1},\dots,X_{i_l})
 \end{equation}
If $X_1=\cdots=X_n=X$, then we write $\kappa_m(X)$ for $\kappa_n(X_1,\dots,X_n)$. A remarkable property of the free cumulants is the following:
\begin{prop}\label{vanishcumulants}
 If  $\{\ca_i\}_{i\in I}$ is a family of  unital subalgebras  of $\ca$, then the following statements are equivalent:
 \begin{enumerate}
 \item[(i)] $\{\ca_i\}_{i\in I}$ are free independent.
 \item[(ii)] For all $a_k\in\ca_{\i(k)}$ ($k=1,\dots,n$ and $i(k)\in I$ we have that
   $\kappa_n(a_1,\dots, a_n)=0$ whenever there exist $1\leq l,k\leq n$ with $i(l)\neq i(n)$.
 \end{enumerate}
 \end{prop}

 An immediate consequence is the additive property of the Voiculescu's $R$-transform (see \cite{ns}, \cite{haagerup}, \cite{vdn}):

\begin{prop}\label{Rtransf}
 If, for any $a\in \ca$, we define $\displaystyle{R_a(z)=\sum_{n=1}^n\kappa_n(X)z^n}$,
 then, for $X,Y$ free, we have that $R_{X+Y}=R_X+R_Y$.
 \end{prop}


  By a non-crossing linked partition $\pi$ of the ordered set
 $\{1,2,\dots,n\}$ we will understand a collection $B_1,\dots, B_k$
 of subsets of $\{1,2,\dots,n\}$, called blocks, with the following
 properties:
 \begin{enumerate}
 \item[(a)]$\displaystyle{\bigcup_{l=1}^kB_l=\{1,\dots,n\}}$
 \item[(b)]$B_1,\dots,B_k$ are non-crossing, in the sense that there
 are no two blocks $B_l, B_s$ and $i<k<p<q$ such that $i,p\in B_l$
 and $k,q\in B_s$.
 \item[(c)] for any $1\leq l,s\leq k$, the intersection $B_l\bigcap
 B_s$ is either void or contains only one element. If
 $\{j\}=B_i\bigcap B_s$, then $|B_s|, |B_l|\geq 2$ and $j$ is the
 minimal element of only one of the blocks $B_l$ and $B_s$.
\end{enumerate}

We will use the notation $s(\pi)$ for the set of all $1\leq k\leq n$ such that there are no blocks of $\pi$ whose minimal element is $k$. A block $B=i_1<i_2<\dots <i_p$ of $\pi$ will be called \emph{exterior} if there is no other block
$D$ of $\pi$ containing two elements $l,s$ such that $l=i_1$ or $l<
i_1<i_p<s$. The set of all non-crossing linked partitions on $\{1,\dots, n\}$
 will be denoted by $NCL(n)$.

\textbf{Example 2}: Below is represented graphically the
non-crossing linked
partition\\
 $\pi=(1,4,6,9), (2,3), (4,5),(6,7,8),(10,11),
(11,12)$ $\in NCL(12)$. Its exterior blocks are $(1, 4, 6, 9)$ and $(10, 11)$.

 \setlength{\unitlength}{.13cm}
 \begin{equation*}
 \begin{picture}(10,8)

\put(-16,3){\circle*{1}}

\put(-12,3){\circle*{1}} \put(-12,3){\line(0,1){3}}

\put(-8,3){\circle*{1}}\put(-8,3){\line(0,1){3}}

\put(-4,3){\circle*{1}}

\put(0,3){\circle*{1}}\put(0,3){\line(0,1){3}}

\put(4,3){\circle*{1}}

\put(8,3){\circle*{1}}\put(8,3){\line(0,1){3}}

\put(12,3){\circle*{1}}\put(12,3){\line(0,1){3}}

\put(16,3){\circle*{1}}

\put(20,3){\circle*{1}}

\put(24,3){\circle*{1}}

\put(28,3){\circle*{1}}\put(28,3){\line(0,1){3}}

\put(-12,6){\line(1,0){4}}

\put(8,6){\line(1,0){4}}

\put(-4,3){\line(4,3){4}}

\put(4,3){\line(4,3){4}}

\put(24,3){\line(4,3){4}}


\put(-16.5,0.2){1}

\put(-12.5,0.2){2}

\put(-8.5,0.2){3}

\put(-4.5,0.2){4}

\put(-0.5,0.2){5}

\put(3.5,0.2){6}

\put(7.5,0.2){7}

\put(11.5,0.2){8}

\put(15.5,0.2){9}

\put(19,0.2){10}

\put(23,0.2){11}

\put(27,0.2){12}

\linethickness{.55mm}

\put(-16,3){\line(0,1){5}}

\put(-16,8){\line(1,0){32}}

\put(20,8){\line(1,0){4}}

\put(20,3){\line(0,1){5.1}}

\put(24,3){\line(0,1){5.1}}

\put(16,3){\line(0,1){5}}

\put(4,3){\line(0,1){5.1}}

\put(-4,3){\line(0,1){5}}

 \end{picture}
 \end{equation*}

 Let $\ca^{\circ}=\ca\setminus Ker \phi$. Using non-crossing linked partitions, we define the \\ $t$-\emph{coefficients} $\{t_n\}_{n=0}^\infty$
 as
the mappings
\[ t_n:\ca\times\left(\ca^\circ\right)^n\lra\mathbb{C}\]
 given by the following recurrence:
 \begin{equation}\label{tcoeff}
 \phi(X_1\cdots X_n)=\sum_{\pi\in NCL(n)}\Large(\prod_{\substack{B=\text{block in } \pi\\B=(i_1,\dots,i_l)}}
 t_{l-1}(X_{i_1},\dots,X_{i_l})\cdot \prod_{k\in s(\pi)}t_0(X_k)\Large).
 \end{equation}
To simplify the writing we will use the shorthand notations $t_\pi[X_1,\dots, X_n]$ for the summing term of the right-hand side of (\ref{tcoeff}),  and $t_\pi[X]$, respectively $t_n(X)$ for $t_\pi(X, \dots, X)$, respectively $t_n(X,\dots, X)$.

\begin{remark}\emph{The mappings $t_n$ are well-defined. Indeed, $t_{n-1}(X_1,\dots, X_{n})$ appears only once and has a non-zero coefficient in the right hand side of (\ref{tcoeff}), namely in}
\[t_{\mathbbm{1}_n}[X_1,\dots, X_n]=t_{n-1}(X_1,\dots, X_{n})\prod_{l=2}^nt_0(X_l)\]
\emph{where $\mathbbm{1}_n$ is the partition with a single block} $(1, 2, \dots, n)$.

\emph{Also, while the free cumulants are multilinear, the $t$-coefficients have the property}
\[ t_n(c_0X_0,c_1X_1,\dots,c_nX_n)=c_0t_n(X_0,X_1,\dots, X_n)\]
\emph{for all} $c_0\in \mathbb{C}, c_1, \dots, c_n\in\mathbb{C}^\ast$.
\emph{The above relation is a immediate consequence of (\ref{tcoeff}), since in each therm of the right-hand side of (\ref{tcoeff}), there is exactly one factor containing $X_0$, on the first position in its block, and exactly two factors containing $X_j$ ($j\geq 1$), among which only one has $X_j$ in the first position.}
\end{remark}

 An alternate form of (\ref{tcoeff}) is described below. For $1\leq k\leq n$ and $\pi\in NCL(n)$, we define
 \[
 t_{[k,\pi]}(X_1,\dots, X_n)=
  \left\{
  \begin{array}{ll}
  t_{s}(X_k, X_{i(1)},\dots, X_{i(s)}) & \text{if $k$ is the minimal element}\\
   &\text{of the block }(k, i(1),\dots, i(s))\\
  t_0(X_k) & \hspace{-2cm}\text{if $k$ is not the minimal element of any block.}
  \end{array}
 \right.
 \]
  Then
  \begin{eqnarray*}
  t_\pi(X_1,\dots, X_n)
  &=&
  \prod_{k=1}^n t_{[k,\pi]}(X_1,\dots, X_n)\\
  \phi(X_1\cdots X_n)
   &=&
   \sum_{\pi\in NCL(n)}t_\pi(X_1,\dots, X_n).
  \end{eqnarray*}

 As shown in \cite{dykema} and \cite{mpttransf}, the $T$-transform, defined via
$\displaystyle \displaystyle{T_X=\sum_{n=0}^\infty t_n(X)z^{n}},$
  has the property that $T_{XY}=T_XT_Y$ for all $X,Y$ free elements of $\ca^\circ$. Notable is also the role of $NCL(n)$ in defining a conditionally free version of the $T$-transform (see \cite{mpttransf}). In the following sections we will discuss the lattice structure of $NCL(n)$, prove a property similar to Proposition \ref{vanishcumulants} for the $t$-coefficients and give a proof for the multiplicative property of the $T$-transform based on the connection between $NCL(n)$, the Kreweras complement on $NC(n)$ and planar rooted trees.

 \section{The lattice $NCL(n)$}

 On $NCL(n)$ we define a order
 relation by saying that $\pi\su \sigma$ if for any block $B$
 of $\pi$ there exist $D_1,\dots, D_s$ blocks of $\sigma$ such that
 $\displaystyle {B=D_1\cup\dots\cup D_s}$.
  With respect to the order relation $\su$, the set $NCL(n)$ is a lattice.
 The maximal,
 respectively the minimal element are $\mathbbm{1}_n=(1,2,\dots,n)$ and  $0_n=(1),(2),\dots,(n)$. Note also that $NC(n)$ is
 a sublattice of $NCL(n)$.

 We say that $i$ and $j$ are \emph{connected }in $\pi\in NCL(n)$ if there
 exist $B_1,\dots,B_s$ blocks of $\pi$ such that $i\in B_1$, $j\in
 B_s$ and $B_k\cap B_{k+1}\neq \varnothing$, $1\leq k\leq s-1$.

 To $\pi\in NCL(n)$ we assign the partition $c(\pi)\in
 NC(n)$ defined as follows: $i$ and $j$ are in the same block of
 $c(\pi)$ if and only if they are connected in $\pi$. (I. e. the blocks of $c(\pi)$ are exactly the connected components of $\pi$.) We
 will use the notation
 \[
 [c(\pi)]=\{\sigma\in NCL(n):
 c(\sigma)=c(\pi)\}.
 \]

 In
 the above Example 2, we have that 5 and 8 as well as 10 and 12 are
 connected. More precisely,  $c(\pi)=(1, 4, 5, 6, 7, 8, 9), (2, 3),  (10, 11, 12)$.

 From the definition of the order relation $\su$, we have that, for every $\gamma \in NC(n)$, $[\gamma]$ is a sublattice of
  $NCL(n)$ and its maximal element is $\gamma$. Moreover,
 if $\gamma$ has the blocks $B_1,\dots, B_s$, each $B_l$ of cardinality
  $k_l$, then  we have the following ordered set
  isomorphism:

  \begin{equation}\label{factorization}
[c(\pi)] \simeq
[\mathbbm{1}_{k_1}]\times\cdots\times[\mathbbm{1}_{k_s}]
  \end{equation}
The above factorization has two immediate consequences:

\begin{prop}\label{cum-tcoeff}
For any positive integer $n$ and any $X_1,\dots, X_n\in\ca$ we have that
\[
\kappa_n(X_1,\dots,X_n)=\sum_{\pi\in[\mathbbm{1}_n]}t_\pi[X_1,\dots, X_n]
\]
\end{prop}
\begin{proof}
For $n=1$ the assertion is clear, since, by definition,
 $$\phi(X_1)=\kappa_1(X_1)=t_0(X_1).$$
If $n>1$, note first that
\[
\sum_{\pi\in NCL(n)}t_\pi[X_1,\dots,X_n]=\sum_{\gamma\in NC(n)}\sum_{\pi\in[\gamma]}t_\gamma[X_1,\dots,X_n].
\]
 Also, if $\pi\in NCL(n)$ has the connected components $B_1,\dots, B_k$ such that each
 $B_l=(i_{l,1},\dots, i_{l,s(l)})$, then
 \[
 t_\pi[X_1,\dots,X_n]=\prod_{l=1}^kt_{\pi_{|B_l}}[X_{i_{i,1}}, \dots, X_{i_{l,s(l)}}],
 \]
where  $\pi_{|B_l}$ denotes the restriction of $\pi$ to the set $B$.

Since the blocks of $c(\pi)$ are by definition the connected components of $\pi$, the relation (\ref{tcoeff}) becomes:
\begin{eqnarray*}
\phi(X_1\cdots X_n)
&=&
\sum_{\gamma\in NC(n)}\prod_{\substack{B=\text{block in } \gamma\\ B=(i_1,\dots, i_s)}}\ \sum_{\pi\in[\gamma]}t_{\pi_{|B}}[X_{i_1},\dots, X_{i_s}],
\end{eqnarray*}
 and the factorization (\ref{factorization}) gives:
 \begin{eqnarray*}
\phi(X_1\cdots X_n)
&=&
\sum_{\gamma\in NC(n)}\prod_{\substack{B=\text{block in } \gamma\\ B=(i_1,\dots, i_s)}}\ \sum_{\sigma\in[\mathbbm{1}_s]}t_{\sigma}[X_{i_1},\dots, X_{i_s}].
\end{eqnarray*}

The conclusion follows now utilizing (\ref{cumulants}) and induction on $n$.

\end{proof}

\begin{prop}\label{vanishtcoeff}\emph{(Characterization of freeness in terms of $t$-coefficients)}\\
 If  $\{\ca_i\}_{i\in I}$ is a family of  unital subalgebras  of $\ca$, then the following statements are equivalent:
 \begin{enumerate}
 \item[(i)] $\{\ca_i\}_{i\in I}$ are free independent.
 \item[(ii)] For all $a_k\in\ca_{\i(k)}$ ($k=1,\dots,n$) and $i(k)\in I$ we have that
   \[
   t_{n-1}(a_1,\dots, a_n)=0
   \]
    whenever there exist $1\leq l,k\leq n$ with $i(l)\neq i(k)$.
 \end{enumerate}
\end{prop}
\begin{proof}
It suffices to show the equivalence between \ref{vanishtcoeff}(ii) and \ref{vanishcumulants}(ii).

Let first suppose that \ref{vanishtcoeff}(ii) holds true. Choose a positive integer $n$ and $\pi\in [\mathbbm{1}_n]$. If not all  $X_k$ are coming from the same subalgebra $\ca_j$, since $1,\dots, n$ are all connected in $\pi$, there is a block $B=(i_1,\dots, i_s)$  such that $\{X_{i_1},\dots, X_{i_s}\}$ contains elements from different subalgebras, therefore $t_{s-1}(X_{i_1},\dots, X_{i_s})=0$, hence $t_\pi[X_1,\dots, X_n]=0$ and \ref{cum-tcoeff} implies that $\kappa_n(x_1,\dots, X_n)=0$, i. e. \ref{vanishcumulants}(ii).

Suppose now that \ref{vanishcumulants}(ii) holds true. Proposition \ref{cum-tcoeff}implies that $\kappa_2=t_1$, so \ref{vanishtcoeff}(ii) is true for $n=2$. An inductive argument on $n$ will complete the proof:

 Proposition \ref{cum-tcoeff} also implies that
\begin{eqnarray*}
\kappa_n(X_1,\dots, X_n)=t_{n-1}(X_1,\dots,X_n)+
\sum_{\substack{\pi\in[\mathbbm{1}_n]\\ \pi\neq \mathbbm{1}_n}}t_\pi[X_1,\dots, X_n].
\end{eqnarray*}

If not all  $X_k$ are coming from the same subalgebra $\ca_j$, then the left-hand side of the above equation cancels from \ref{vanishcumulants}(ii) and so does the second term of the right-hand side, from the induction hypothesis, so q.e.d..
\end{proof}

\section{planar trees and the multiplicative property of the $T$-transform}

In this section we will give a combinatorial proof for the multiplicative property of the $T$-transform, that is
 \begin{equation}\label{Tmultipl}
 T_{XY}=T_XT_Y \ \text{whenever } X \text{and } Y \text{are free elements from } \ca^\circ.
 \end{equation}

 The proof will consist mainly in describing certain bijections between
 $[\mathbbm{1}_n]$ and the set $\mathfrak{T}(n)$ of \emph{planar trees}
 with $n$ vertices, respectively between $NC(n)$ and the set $\mathfrak{B}(n)$
  of \emph{bicolor planar trees} with $n$ vertices.

  \subsection{Non-crossing linked partitions and planar trees}

  By an \emph{elementary planar tree} we will understand a graph with
   $m\geq 1$ vertices,  $v_1,v_2,\dots, v_m$, and $m-1$ edges, or branches,
    connecting $v_1$ (called \emph{root}) to the  vertices $v_2, \dots, v_m$ (called \emph{offsprings}). By convention, a single vertex (with no offsprings) will be also considered an elementary planar tree.

  A \emph{planar tree} will be seen as consisting in a finite number
  of \emph{levels}, such that:
  \begin{enumerate}
  \item[-]first level consists in a single elementary planar tree,
  whose root will be also the root of the planar tree;
   \item[-]the $k$-th level will consist in a set of elementary
    planar trees such that their roots are among the offsprings of the $k-1$-th level.
      \end{enumerate}
  The set of elementary trees composing the planar tree $A$ will be denoted by  $E(A)$.

 Below are represented graphically the elementary planar tree $T_1$ and the 2-level planar tree $T_2$:

 \setlength{\unitlength}{.15cm}
 \begin{equation*}
 \begin{picture}(22,18)

\put(-22,8){\circle*{1}} \put(-22,8){\line(2,3){4}}

\put(-18,14){\circle*{1}}

\put(-18,8){\circle*{1}}\put(-18,8){\line(0,1){6}}

\put(-14,8){\circle*{1}} \put(-14,8){\line(-2,3){4}}

\put(-19,4){$T_1$}

\put(8,16){\circle*{1}}

\put(5,11){\circle*{1}} \put(5,11){\line(3,5){3}}


\put(11,11){\circle*{1}} \put(11,11){\line(-3,5){3}}

\put(8,6){\circle*{1}} \put(8,6){\line(3,5){3}}

\put(14,6){\circle*{1}} \put(14,6){\line(-3,5){3}}

\put(5,2){$T_2$}

\multiput(17,6)(4,0){3}{\line(1,0){2}}
\multiput(16,11)(4,0){3}{\line(1,0){2}}
\multiput(15,16)(4,0){3}{\line(1,0){2}}
\multiput(17,2)(4,0){3}{\line(1,0){2}}

\put(26,3){level 3}

\put(26,8){level 2}

\put(26,13){level 1}

 \end{picture}
 \end{equation*}

  We will need  to consider on the vertices of a planar tree the order relation ( similar to the  ``\emph{left depth first}'' order from \cite{eap}) given by:
   \begin{enumerate}
   \item[(i)]roots are less than their offsprings;
   \item[(ii)] offsprings of the same root are ordered from left to right;
    \item[(iii)]if $v$ is less that $w$, then all the offsprings of $v$ are smaller than any offspring of $w$.
    \end{enumerate}
        Intuitively, one may understand it as the order in which the vertices are passed by walking along the branches from the root to the right-most vertex, not counting vertices passed more than one time (see the example below).

  \textbf{Example} 3:

   \setlength{\unitlength}{.13cm}
 \begin{equation*}
 \begin{picture}(18,18)

\put(8,16){\circle*{1}}

\put(3,11){\circle*{1}} \put(3,11){\line(1,1){5}}

\put(3,6){\circle*{1}}\put(3,6){\line(0,1){5}}

\put(11,11){\circle*{1}} \put(11,11){\line(-3,5){3}}

\put(8,6){\circle*{1}} \put(8,6){\line(3,5){3}}

\put(14,6){\circle*{1}} \put(14,6){\line(-3,5){3}}

\put(8,18){1}

\put(1,11){2}

\put(2,3.5){3}

\put(12,10.5){4}

\put(7.4,3.5){5}

\put(13.5,3.5){6}
 \end{picture}
 \end{equation*}

 We
construct now the bijection
 $\Theta: [\mathbbm{1}_n]\lra\mathfrak{T}(n)$
  by putting  $\Theta(\pi)$ be the planar tree composed by the elementary trees of vertices
  numbered $(i_1, \dots, i_s)$ (with respect to the above order relation), for each $(i_1, \dots, i_s)$ block of $\pi$.

   More precisely, if $(1,2,i_1,\dots, i_s)$ is the block of $\pi$ containing 1,
   then the first level of $\Theta(\pi)$ is the elementary planar tree of root
    numbered 1 and offsprings numbered $(2,i_1,\dots, i_s)$.
     The second level of $\Theta(\pi)$ will be determined by the blocks
     (if any) having $2,i_1,\dots, i_s$ as first elements etc (see the example below)
     \smallskip

     \includegraphics[width=3in,height=1.1cm]{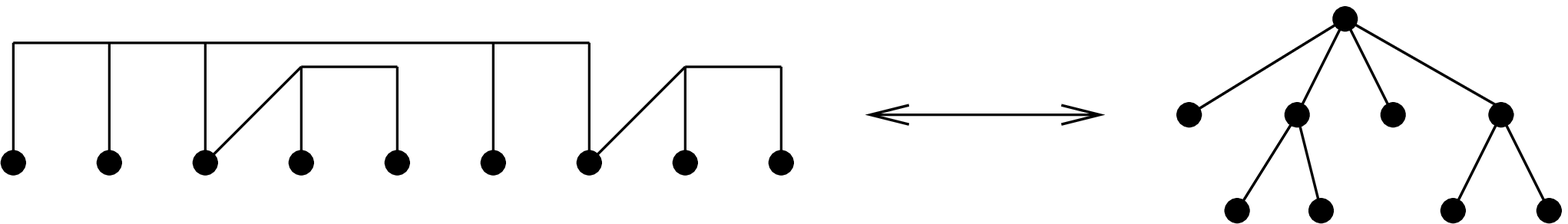}

    \smallskip

     It is easy to see that $\Theta$ is well-defined and injective.
     To show that $\Theta$ is bijective, it suffices to observe that
 $\Theta^{-1}$ is given by assigning to each tree the partition having
 the blocks given by the numbers of the vertices from the constituent
  elementary planar trees.
   Part (a) of the definition of $NCL(n)$ (see Section 1) is
   automatically verified, since the trees have exactly $n$ vertices.
   Part (b) follows from the conditions (ii) and (iii) in the definition of
    the order relation on the vertices, and part (c) from the conditions (i) and (ii).

 If $A_0$ is an elementary planar tree with $n$ vertices and $X\in\ca^\circ$, we define $\cE_X(A_0)=t_{n-1}(X)$. The evaluation $\cE_X$ extends to the set of planar trees by
 \[ \cE_X(A)=\prod_{A_0\in E(A)}\cE_X(A_0).\]

Consequently,
\begin{equation}\label{eq5}
\kappa_n(X)=\sum_{A\in
\mathfrak{T}(n)}\cE_X(A).
\end{equation}

\subsection{The Kreweras complement and bicolor planar trees}

For $\gamma\in NC(n)$, its \emph{Kreweras complement} $\kr(\gamma)$ is defined as follows. We consider the additional numbers $\overline{1}, \dots, \overline{n}$ forming the ordered set
\[ 1,\overline{1}, 2, \overline{2}, \dots, n, \overline{n}.\]
$\kr(\gamma)$ is defined to be the biggest element $\gamma^\prime\in NC( \overline{1}, \dots, \overline{n})\cong NC(n)$ such that
 \[
  \gamma \cup \gamma^\prime\in NC(1,\overline{1}, 2, \overline{2}, \dots, n, \overline{n}).
  \]
   The total number of blocks in $\gamma$ and $\kr(\gamma)$ is $n+1$ (see \cite{ns}, \cite{kreweras}).
   The Kreweras complement appears in the following corollary of Proposition \ref{vanishcumulants}:

   \begin{prop}\label{cumulantprod}
   If $X, Y$ are free elements of $\ca$, then
   \[ \kappa_n(XY)=\sum_{\gamma\in NC(n)}\kappa_{\gamma}[X]\kappa_{\kr(\gamma)}[Y].
   \]
   \end{prop}

   Let $NC_s(2n)$ be the set of all $\gamma\in NC(2n)$ such that  elements from the same block of $\gamma$ have the same parity and $\gamma_+=Kr(\gamma_-)$, where
   \begin{eqnarray*}
   \gamma_{+}&=&\gamma_{|\{2,4,\dots, 2n\}}\\
  \gamma_-&=&\gamma_{|\{1,3,\dots, 2n-1\}}.
   \end{eqnarray*}
    Denote also $NCL_S(2n)=\{\pi\in NCL(2n): c(\pi)\in NC_s(2n)\}$.
   With the above notations, the relation from Proposition \ref{cumulantprod} becomes:
   \begin{eqnarray}
   \kappa_n(XY)&=&\sum_{\gamma\in NC_s(2n)}\kappa_{\gamma_-}[X]\kappa_{\gamma_+}[Y]\label{ncsprod}
   \end{eqnarray}

  For $\pi\in NC_S(2n)$, we will say that the blocks with odd elements are of color 1 and the ones with even elements are of color 0. Note that $\pi\in NC_S(2n)$ if and only if $\pi$ has exactly 2 exterior blocks, one of color 1 and one of color 0 and if $i_1$ and $i_2$ are two consecutive elements from the same block, then $\pi_{|(i_1+1, \dots, i_2-1)}$ has exactly one exterior block, of different color than the one containing $i_1$ and $i_2$.

  We will represent blocks of color 1 by solid lines and blocks of color 0 by dashed lines:
  \medskip

\includegraphics[width=3in,height=.9cm]{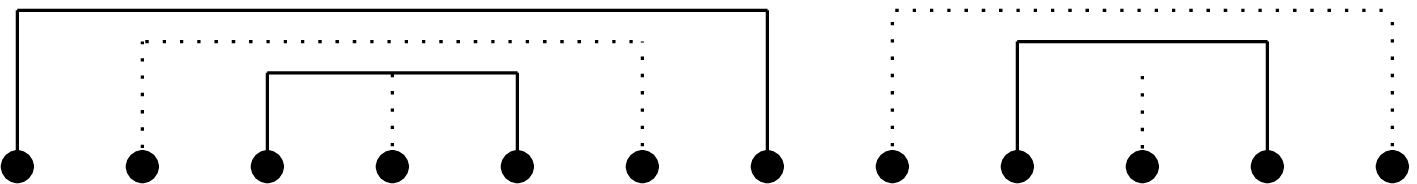}

\smallskip

  In the remaining part of this subsection we will define the set $\mathfrak{B}(n)$ of bicolor planar trees and construct a bijection $\Lambda: NC_S(2n)\lra\mathfrak{B}(n)$.

  A \emph{bicolor elementary  planar tree} is an elementary tree together with a mapping from its offsprings to $\{0, 1\}$ such that the offsprings whose image is 1 are smaller (in the sense of Section 3.1) than the offsprings of image 0. Branches toward offsprings of color ), respectively 1, will be also said to be of color 0, respectively 1. We will represent by solid lines the branches of color 1 and by dashed lines the branches of color 0. The set of all bicolor planar trees with $n$ vertices will be denoted by $\mathfrak{EB}(n)$. Below is the graphical representation of $\mathfrak{EB}(4)$:
  \smallskip

\includegraphics[width=3in,height=.9cm]{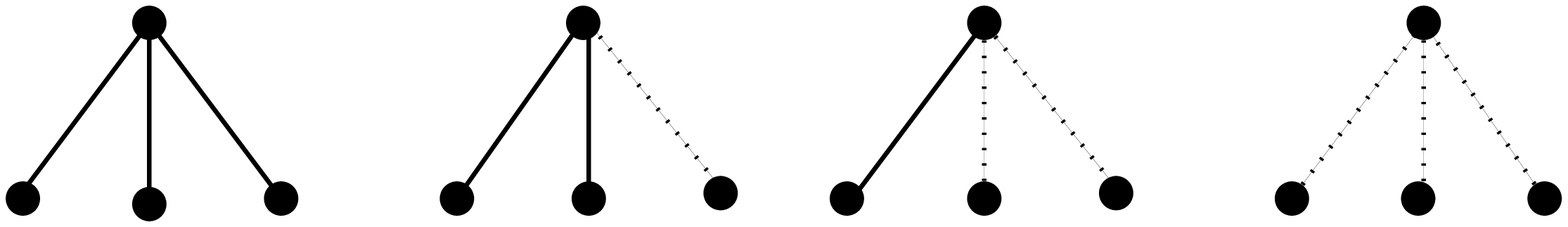}

\smallskip
  A \emph{bicolor planar tree} is a planar tree whose constituent elementary trees are all bicolor; the set of all bicolor planar trees will be denoted by $\mathfrak{B}(n)$.



   Given $\pi\in NCL_S(2n)$, we construct $\Lambda(\pi)\in\mathfrak{B}(n)$ as follows:

   \begin{enumerate}
   \item[-] If $(i_1,\dots, i_s)$ and $(j_1,\dots, j_p$ are the two exterior blocks of $\pi$, then the first level of $\Lambda(\pi)$ is an elementary tree with $s-1+p-1$ offsprings,  the first $s-1$ of color 1, corresponding to $(i_2,\dots, i_s)$, in this order, and the last $p-1$ of color 0, corresponding to $(j_2,\dots, i_p)$, in this order.

   \item[-]Suppose that $i_1$ and $i_2$ are consecutive elements in a block of $\pi$ already represented in an elementary tree of $\Lambda(\pi)$, that $\pi_{}$ has the exterior block $(j_1,\dots, j_p)$ and that $i_2$ id the minimal element of the block $(i_2, d_1,\dots, d_r)$. The the blocks $B=(j_1,\dots, j_p)$  and $D=(i_2, d_1,\dots, d_r)$ will have different colors. They will be then represented by an elementary tree of vertex corresponding to $i_2$ (the block of $i_1$ and $i_2$ has been already represented from the hypothesis), and with $p-1+k$ offsprings, keeping the colors of the blocks $B$ and $D$, the ones of color 1 placed before the ones of color 0.

   \end{enumerate}
Note that the mapping $\Lambda$ is bijective, the inverse is constructing reversing the steps above.

\includegraphics[width=8cm,height=2cm]{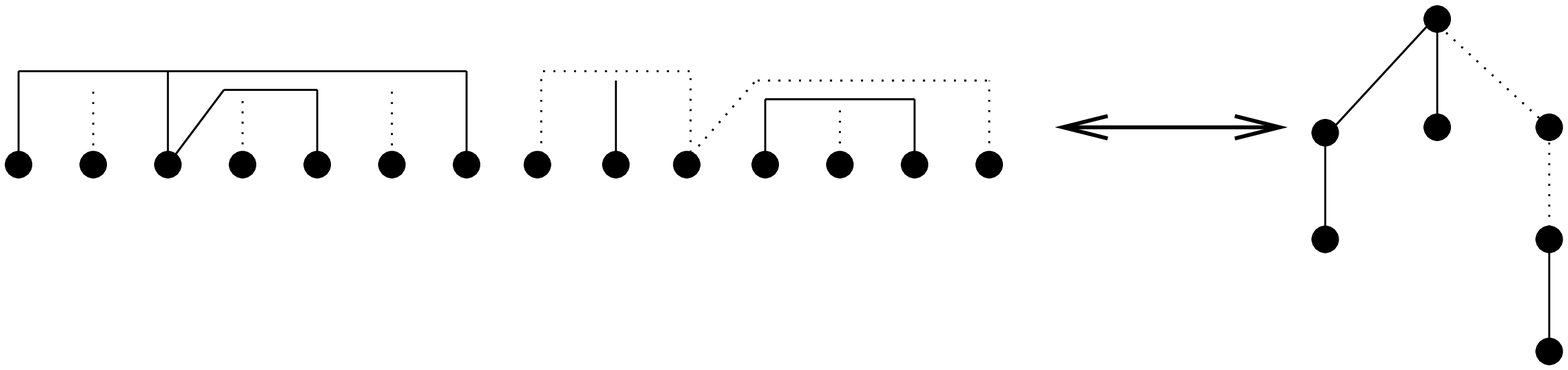}

  Fix $X,Y\in\ca^\circ$, free. If
  $B_0\in\mathfrak{EB}(n)$
   has $k$ offsprings of color 1 and $n-k-1$ offsprings of color 0, then we define
\[\omega_{X,Y}(B_0)=t_k(X)t_{n-k-1}(Y).\]
 The functional $\omega_{X,Y}$ extends to $B\in\mathfrak{B}(n)$ via
 \[ \omega_{X,Y}(B)=\prod_{B_0\in E(B)}\omega_{X,Y}(B_0).\]
For $\pi\in NC_S(2n)$ the definition of the mapping $\Lambda$ gives
\begin{equation}\label{eq10}
\kappa_{\pi_-}[X]\kappa_{\pi_+}[Y]=\omega_{X,Y}(\Lambda(\pi))
\end{equation}

\begin{thm} If $X,Y$ are free elements from $\ca^\circ$, then $T_{XY}=T_XT_Y$.
\end{thm}
\begin{proof}
 We need to show that, for all $m\geq 0$
 \begin{equation}\label{final}
 t_m(XY)=\sum_{k=0}^mt_k(X)t_{m-k}(Y))
 \end{equation}
 For $m=0$, the assertion is trivial. Suppose (\ref{final}) true for $m\leq n-1$.
  Let $A_m$ be the elementary planar tree with $m$ vertices. In terms of planar trees, the induction hypothesis is written as
 \begin{equation}\label{elemsum}
 \cE_{XY}(A_m)=\sum_{B\in\mathfrak{EB}(n)}\omega_{X,Y}(B).\end{equation}

For example,
\setlength{\unitlength}{.15cm}
 \begin{equation*}
 \begin{picture}(28,7)

 \put(-21,4){$\cE_{XY}($}

\put(-14,3){\circle*{1}}

\put(-12,3){\circle*{1}}

\put(-10,3){\circle*{1}}

\put(-12,7){\circle*{1}}

\put(-14,3){\line(1,2){2}}

\put(-12,3){\line(0,1){4}}

\put(-10,3){\line(-1,2){2}}

\put(-9,4){$)$}

\put(-7,4){$=$}

\put(-5,4){$\omega_{XY}($}

\put(2,3){\circle*{1}}

\put(4,3){\circle*{1}}

\put(6,3){\circle*{1}}

\put(4,7){\circle*{1}}

\put(2,3){\line(1,2){2}}

\put(4,3){\line(0,1){4}}

\put(6,3){\line(-1,2){2}}

\put(7,4){$)$}

\put(9,4){$+$}

\put(11,4){$\omega_{XY}($}

\put(18,3){\circle*{1}}

\put(20,3){\circle*{1}}

\put(22,3){\circle*{1}}

\put(20,7){\circle*{1}}

\put(18,3){\line(1,2){2}}

\put(20,3){\line(0,1){4}}

\multiput(20,7)(.25,-.5){8}{\line(0,1){.2}}

\put(23,4){$)$}

\put(25,4){$+$}

\put(27,4){$\omega_{XY}($}

\put(34,3){\circle*{1}}

\put(36,3){\circle*{1}}

\put(38,3){\circle*{1}}

\put(36,7){\circle*{1}}

\put(34,3){\line(1,2){2}}

\multiput(36,3)(0,.5){8}{\line(0,1){.2}}

\multiput(36,7)(.25,-.5){8}{\line(0,1){.2}}

\put(39,4){$)$}

\put(41,4){$+$}

\put(43,4){$\omega_{XY}($}

\put(50,3){\circle*{1}}

\put(52,3){\circle*{1}}

\put(54,3){\circle*{1}}

\put(52,7){\circle*{1}}

\multiput(52,7)(-.25,-.5){8}{\line(0,1){.2}}

\multiput(52,3)(0,.5){8}{\line(0,1){.2}}

\multiput(52,7)(.25,-.5){8}{\line(0,1){.2}}

\put(55,4){$)$}

 \end{picture}
 \end{equation*}

 The relations
 (\ref{eq5}) and (\ref{eq10}) give:
 \begin{eqnarray}
 \sum_{A\in\mathfrak{T}(n)}\cE_{XY}(A)
  &=&\kappa_n(X) \nonumber\\
  &=&\sum_{\pi\in NC_S(2n)}\kappa_{\pi_-}[X]\kappa_{\pi_+}[Y]\nonumber\\
  &=&\sum_{\pi\in NC_S(2n)}\omega_{X,Y}(\Lambda(\pi))\nonumber\\
  &=&\sum_{B\in\mathfrak{B}(n)}\omega_{X,Y}(B)\label{eq7}.
  \end{eqnarray}

 All non-elementary trees from $\mathfrak{T}(n)$ consists on elementary trees with less than $n$ vertices. The relation (\ref{elemsum}) implies that the image under $\cE_{XY}$ of any such tree is the sum of the images under $\omega_{XY}$ of its colored versions. Hence
 \begin{equation}\label{eq11}
 \sum_{\substack{A\in\mathfrak{T}(n)\\A\neq A_n}}\cE_{XY}(A)
  =
   \sum_{\substack{B\in\mathfrak{B}(n)\\B\in\mathfrak{EB}(n)}}\omega_{X,Y}(B)
 \end{equation}

 Finally (\ref{eq11}) and (\ref{eq7}) give
 \[ \cE_{XY}(A_n)=\sum_{B\in\mathfrak{B}(n)}\omega_{X,Y}(B)\]
 that is (\ref{final}).
\end{proof}



\begin{thebibliography}{10}



\bibitem{eap}
M. Anshelevich, E. G.  Effros, M. Popa.
\newblock{Zimmermann type cancellation in the free Fa\`{a} di Bruno
algebra}.
\newblock{ J. Funct. Anal.  237 (2006),  no. 1, 76--104.}





\bibitem{dykema}
K. Dykema.
 \newblock{Multilinear function series and transforms in Free Probability
 theory.}
 \newblock{Preprint, arXiv:math.OA/0504361 v2 5 Jun 2005}






 \bibitem{haagerup}
 U. Haagerup.
 \newblock{On Voiculescu's $R$- and $S$-transforms for Free non-commuting Random
 Variables.}
 \newblock{Fields Institute Communications, vol. 12(1997), 127--148}


\bibitem{kreweras}
G. Kreweras.
\newblock{ Sur les partitions non-croisees d'un cycle.}
\newblock{Discrete Math. 1 (1972), pp. 333–-350}



\bibitem{ns}
A. Nica, R. Speicher.
\newblock{Lectures on the Combinatorics of the Free Probability.}
\newblock London mathematical Society Lecture Note Series 335,
Cambridge University Press 2006








\bibitem{booleanmvp}
 M. Popa.
 \newblock{A new proof for the multiplicative
 property of the boolean cumulants with applications
 to operator-valued case}
 \newblock{arXiv:0804.2109 }

 \bibitem{mpttransf}
Popa, Mihai, J. C. Wang. \newblock{On multiplicative conditionally free
convolution}, arXiv:0805.0257

\bibitem{speicher1}
R. Speicher.
\newblock Combinatorial Theory of the Free Product with amalgamation
and Operator- Valued Free Probability Theory.
\newblock Mem. AMS, Vol 132, No 627 (1998)


 \bibitem{vdn}
 D.V. Voiculescu, K. Dykema, A. Nica.
 \newblock{Free random variables.}
 \newblock{CRM Monograph Series, 1. AMS, Providence, RI, 1992.}







\end{thebibliography}
\end{document}